\documentclass[12pt]{article}

\usepackage{amsmath, epsfig, cite}
\usepackage{amssymb}
\usepackage{amsfonts}
\usepackage{latexsym}
\usepackage{amsthm}

\makeatletter
\renewcommand{\@seccntformat}[1]{{\csname the#1\endcsname}.\hspace{.5em}}
\makeatother

\newtheorem{thm}{Theorem}[section]

\newtheorem{conj}[thm]{Conjecture}

\renewcommand{\qed}{\hfill$\Box$\medskip}

\numberwithin{equation}{section}

\allowdisplaybreaks

\begin{document}

\renewcommand{\thefootnote}{*}

\begin{center}
{\large\bf Factors of some truncated basic hypergeometric series}
\end{center}

\vskip 2mm \centerline{Victor J. W. Guo}
\begin{center}
{\footnotesize School of Mathematical Sciences, Huaiyin Normal University, Huai'an 223300, Jiangsu\\
 People's Republic of China\\[5pt]

{\tt jwguo@hytc.edu.cn}  }
\end{center}

\vskip 0.7cm \noindent{\small{\bf Abstract.}  We prove that certain basic hypergeometric series truncated at $k=n-1$ have
the factor $\Phi_n(q)^2$, where $\Phi_n(q)$ is the $n$-th cyclotomic polynomial. This confirms two recent conjectures
of the author and Zudilin.  We also put forward some conjectures on $q$-congruences modulo $\Phi_n(q)^2$.}



\section{Introduction}
Rodriguez-Villegas  \cite{RV} discovered numerically some remarkable supercongruences on truncated hypergeometric series
related to a Calabi-Yau manifold. The simplest supercongruence of Rodriguez-Villegas is: for any odd prime $p$,
\begin{align}
\sum_{k=0}^{p-1}\frac{{2k\choose k}^2}{16^k} \equiv
(-1)^{(p-1)/2}\pmod{p^2}.
\label{eq:RV}
\end{align}
It has caught the interests of many authors (see \cite{CLZ, GZ14, Mortenson1, SunZH, SunZW, Tauraso1,Tauraso2}). For example, Guo and Zeng \cite{GZ14} proved a $q$-analogue of
\eqref{eq:RV}:
\begin{align}
\sum_{k=0}^{p-1}\frac{(q;q^2)_k^2}{(q^2;q^2)_k^2} \equiv
(-1)^{(p-1)/2}q^{(1-p^2)/4}\pmod{[p]^2}\quad\text{for any odd prime $p$}. \label{eq:GZ-RV}
\end{align}
Here and in what follows,
$(a;q)_n=(1-a)(1-aq)\cdots (1-aq^{n-1})$
is the {\it $q$-shifted factorial}, and $[n]=1+q+\cdots+q^{n-1}$ is the {\it $q$-integer}.
The $q$-congruence \eqref{eq:GZ-RV} has been further generalized by Guo, Pan, and Zhang \cite{GPZ},
Ni and Pan \cite{NP}, and Guo \cite{Guo-par}. A slight generalization of \eqref{eq:GZ-RV}
can be stated as follows (see \cite{Guo-par,NP}):
\begin{align*}
\sum_{k=0}^{n-1}\frac{(q;q^2)_k^2}{(q^2;q^2)_k^2} \equiv
(-1)^{(n-1)/2}q^{(1-n^2)/4}\pmod{\Phi_n(q)^2}\quad\text{for positive odd $n$},
\end{align*}
where $\Phi_n(q)$ is the $n$-th {\it cyclotomic polynomial} in $q$.

Recently, the author and Zudilin \cite[Conjecture 5.3]{GuoZu} conjectured that, for $d\geqslant 3$ and $n\equiv -1\pmod{d}$,
\begin{align}
\sum_{k=0}^{n-1}\frac{(q;q^d)_k^d q^{dk}}{(q^d;q^d)_k^d} \equiv 0 \pmod{\Phi_n(q)^2}.  \label{eq:d-1}
\end{align}
They \cite[Conjecture 5.4]{GuoZu} also conjectured that, for
$n,d\geqslant 2$ and $n\equiv 1\pmod{d}$,
\begin{align}
\sum_{k=0}^{n-1}\frac{(q^{-1};q^d)_k^d q^{dk}}{(q^d;q^d)_k^d} \equiv
0 \pmod{\Phi_n(q)^2}.  \label{eq:d-2}
\end{align}

In this paper, we shall confirm the above two conjectures. It turns out that much more is true and we shall prove the following unified generalization of \eqref{eq:d-1} and \eqref{eq:d-2}.
\begin{thm}\label{main-1}
Let $d\geqslant 2$ be an integer. Let $r\leqslant d-2$ be an integer such that $\gcd(r,d)=1$. Then, for all positive integers $n$
with $n\equiv -r\pmod{d}$ and $n\geqslant d-r$, we have
\begin{align}
\sum_{k=0}^{n-1}\frac{(q^r;q^d)_k^d q^{dk}}{(q^d;q^d)_k^d} \equiv 0 \pmod{\Phi_n(q)^2}. \label{eq:main}
\end{align}
\end{thm}

It is clear that if $d\geqslant 3$ and $r=1$, then the congruence \eqref{eq:main} reduces to \eqref{eq:d-1}, while if $d\geqslant 2$ and $r=-1$,
then the congruence \eqref{eq:main} leads to \eqref{eq:d-2}.

For $d=2$ and $r=-1$, we have the following stronger result and conjecture.
\begin{thm} \label{main-2}
Let $n>1$ be a positive odd integer. Then
\begin{align}
\sum_{k=0}^{n-1}\frac{(q^{-1};q^2)_k^2 q^{2k}}{(q^2;q^2)_k^2}     &\equiv 0 \pmod{[n]\Phi_n(q)}, \label{eq:main-2} \\[5pt]
\sum_{k=0}^{(n+1)/2}\frac{(q^{-1};q^2)_k^2 q^{2k}}{(q^2;q^2)_k^2} &\equiv 0 \pmod{[n]\Phi_n(q)}. \label{eq:main-3}
\end{align}
\end{thm}

\begin{conj}The congruences \eqref{eq:main-2} and \eqref{eq:main-3} still hold modulo $[n]^2$.
\end{conj}

We shall also give some similar results, such as
\begin{thm}\label{thm-1}
Let $n>1$ be a positive integer. Then
\begin{align}
\sum_{k=0}^{n-1}\frac{(q,q,q^4;q^6)_k q^{6k}}{(q^6;q^6)_k^3}
&\equiv 0 \pmod{\Phi_n(q)^2}\quad\text{if $n\equiv 5\pmod{6}$,} \label{main-6-5}\\[5pt]
\sum_{k=0}^{n-1}\frac{(q^{-1}, q^{-1},q^{-4};q^6)_k q^{6k}}{(q^6;q^6)_k^3}
&\equiv 0 \pmod{\Phi_n(q)^2} \quad\text{if $n\equiv 1\pmod{6}$}.  \label{main-6-1}
\end{align}
\end{thm}

We shall prove Theorems \ref{main-1} and \ref{thm-1}  by using the {\it creative microscoping} method developed by the author and Zudilin \cite{GuoZu}. That is to say, to prove
a $q$-supercongruence modulo $\Phi_n(q)^2$, it is more convenient to establish its generalization with an additional parameter $a$ so that the generalized congruence
holds modulo $(1-aq^n)(a-q^n)$. The difference here is that we shall add the parameter $a$ in quite a different way for the proof of Theorem \ref{main-1}.
The proof of Theorem \ref{main-2} is based on Theorem \ref{main-1} and borrows some idea from \cite{GuoZu} for proving congruences modulo $[n]$.
We shall give more similar congruences modulo $\Phi_n(q)^2$ in Section 5 and propose some related open problems in the last section.

\section{Proof of Theorem \ref{main-1}}

We first establish the following parametric generalization of Theorem \ref{main-1}.
\begin{thm}
Let $d,r,n$ be given as in the conditions of Theorem \ref{main-1}. Then, modulo $(1-aq^n)(a-q^n)$,
\begin{align}
\sum_{k=0}^{n-1}\frac{(a^{d-1}q^r, a^{d-3}q^r,\ldots, a^2q^r;q^d)_k  (a^{1-d}q^r, a^{3-d}q^r,\ldots, a^{-2}q^r;q^d)_k (q^r;q^d)_k q^{dk}}
{(a^{d-2}q^d, a^{d-4}q^d,\ldots,aq^d;q^d)_k  (a^{2-d}q^d, a^{4-d}q^d,\ldots, a^{-1}q^d;q^d)_k (q^d;q^d)_k } \equiv 0\label{eq:a-1}
\end{align}
if $d$ is odd, and
\begin{align}
\sum_{k=0}^{n-1}\frac{(a^{d-1}q^r, a^{d-3}q^r,\ldots, aq^r;q^d)_k  (a^{1-d}q^r, a^{3-d}q^r,\ldots, a^{-1}q^r;q^d)_k  q^{dk}}
{(a^{d-2}q^d, a^{d-4}q^d,\ldots,q^d;q^d)_k  (a^{2-d}q^d, a^{4-d}q^d,\ldots, q^d;q^d)_k } \equiv 0\label{eq:a-e}
\end{align}
if $d$ is even.
\end{thm}
\begin{proof} Since $\gcd(r,d)=1$ and $n\equiv -r\pmod{d}$, we have $\gcd(d,n)=1$ and so the numbers $d,2d,\ldots (n-1)d$ are all not
divisible by $n$. This means that the denominators of the left-hand sides of \eqref{eq:a-1} and \eqref{eq:a-e} do not contain
the factor $1-aq^n$ nor $1-a^{-1}q^n$.
Hence, for $a=q^{-n}$ or $a=q^n$, the left-hand side of \eqref{eq:a-1} can be written as
\begin{align}
\sum_{k=0}^{\frac{dn-n-r}{d}}\frac{(q^{r-(d-1)n}, q^{r-(d-3)n},\ldots, q^{r-2n};q^d)_k  (q^{(d-1)n+r}, q^{(d-3)n+r},\ldots, q^{2n+r};q^d)_k (q^r;q^d)_k q^{dk}}
{(q^{d-(d-2)n}, q^{d-(d-4)n},\ldots, q^{d-n};q^d)_k  (q^{(d-2)n+d}, q^{(d-4)n+d},\ldots, q^{n+d};q^d)_k (q^d;q^d)_k },\label{eq:a-2}
\end{align}
where we have used the fact that $(q^{r-(d-1)n};q^d)_k=0$ for $k>(dn-n-r)/d$, and by the conditions there holds $0<(dn-n-r)/d\leqslant n-1$.

Let
\begin{equation*}
{n\brack k}={n\brack k}_q=
\frac{(q;q)_n}{(q;q)_k(q;q)_{n-k}}
\end{equation*}
be the {\it $q$-binomial coefficient}. It is easy to see that
\begin{align}
\frac{(q^{r-(d-1)n};q^d)_k q^{dk}}{(q^d;q^d)_k} &=(-1)^k {(dn-n-r)/d\brack k}_{q^d} q^{d{k\choose 2}+(n+r-dn+d)k},\label{qdk-0}\\[5pt]
\frac{(q^{r-(d-3)n};q^d)_k}{(q^{d-(d-2)n};q^d)_k} &=\frac{(q^{d-(d-2)n+dk};q^d)_{(n+r-d)/d}}{(q^{d-(d-2)n};q^d)_{(n+r-d)/d}},\label{qdk-begin}\\[5pt]
\frac{(q^{r-(d-5)n};q^d)_k}{(q^{d-(d-4)n};q^d)_k} &=\frac{(q^{d-(d-4)n+dk};q^d)_{(n+r-d)/d}}{(q^{d-(d-4)n};q^d)_{(n+r-d)/d}},\\[5pt]
&\ \  \vdots \notag \\[5pt]
\frac{(q^{r-2n};q^d)_k}{(q^{d-3n};q^d)_k} &=\frac{(q^{d-3n+dk};q^d)_{(n+r-d)/d}}{(q^{d-3n};q^d)_{(n+r-d)/d}},
\end{align}
and
\begin{align}
\frac{(q^{(d-1)n+r};q^d)_k}{(q^{(d-2)n+d};q^d)_k} &=\frac{(q^{(d-2)n+dk+d};q^d)_{(n+r-d)/d}}{(q^{(d-2)n+d};q^d)_{(n+r-d)/d}},\\[5pt]
\frac{(q^{(d-3)n+r};q^d)_k}{(q^{(d-4)n+d};q^d)_k} &=\frac{(q^{(d-4)n+dk+d};q^d)_{(n+r-d)/d}}{(q^{(d-4)n+d};q^d)_{(n+r-d)/d}},\\[5pt]
&\ \  \vdots \notag \\[5pt]
\frac{(q^{2n+r};q^d)_k}{(q^{n+d};q^d)_k} &=\frac{(q^{2n+dk+d};q^d)_{(n+r-d)/d}}{(q^{n+d};q^d)_{(n+r-d)/d}},\\[5pt]
\frac{(q^r;q^d)_k}{(q^{d-n};q^d)_k} &=\frac{(q^{d-n+dk};q^d)_{(n+r-d)/d} }{(q^{d-n};q^d)_{(n+r-d)/d} }.  \label{qdk-end}
\end{align}
Noticing that the right-hand sides of \eqref{qdk-begin}--\eqref{qdk-end} are all polynomials in $q^{dk}$ of degree $(n+r-d)/d$, and
$$
d{k\choose 2}+(n+r-dn+d)k=d{(dn-n-r)/d-k\choose 2}-d{(dn-n-r)/d\choose 2},
$$
we can write \eqref{eq:a-2} as
\begin{align}
\sum_{k=0}^{\frac{dn-n-r}{d}}(-1)^k q^{d{(dn-n-r)/d-k\choose 2}} {(dn-n-r)/d\brack k}_{q^d}P(q^{dk}), \label{eq:a-p}
\end{align}
where $P(q^{dk})$ is a polynomial in $q^{dk}$ of degree $(n+r-d)(d-1)/d=(dn-n-r)/d-(d-r-1)\leqslant (dn-n-r)/d-1$.

Recall that the finite form of the $q$-binomial theorem (see, for example, \cite[p. 36]{Andrews}) can be written as
\begin{align*}
\sum_{k=0}^{n}(-1)^k {n\brack k}q^{k\choose 2} z^k=(z;q)_n.
\end{align*}
Letting $z=q^{-j}$ and replacing $k$ with $n-k$ in the above equation, we obtain
\begin{align}
\sum_{k=0}^{n}(-1)^k {n\brack k}q^{{n-k\choose 2}+jk}=0\quad\text{for}\ 0\leqslant j\leqslant n-1.  \label{eq:qbino}
\end{align}
This immediately implies that $\eqref{eq:a-2}=\eqref{eq:a-p}=0$. Namely, the congruence \eqref{eq:a-1} holds.

Along the same lines, we can prove the congruence \eqref{eq:a-e}.
\end{proof}

\begin{proof}[Proof of Theorem {\rm\ref{main-1}}]Note that $\Phi_n(q)$ is a factor of $1-q^m$ if and only if $n$ divides $m$.
It follows that the limits of the denominators of \eqref{eq:a-1} and \eqref{eq:a-e} as $a\to1$ are relatively prime to $\Phi_n(q)$,
since $n$ is coprime with $d$.
On the other hand, the limit of $(1-aq^n)(a-q^n)$ as $a\to1$ has the factor $\Phi_n(q)^2$.
Thus, the congruence \eqref{eq:main} follows from the limiting case $a\to1$ of \eqref{eq:a-1} and \eqref{eq:a-e}.
\end{proof}

\section{Proof of Theorem \ref{main-2}}
Letting $d=2$ and $r=-1$ in \eqref{eq:main}, we see that, for odd $n>1$,
\begin{align}
\sum_{k=0}^{n-1}\frac{(q^{-1};q^2)_k^2 q^{2k}}{(q^2;q^2)_k^2}     &\equiv 0 \pmod{\Phi_n(q)^2}, \label{eq:main-2-2} \\[5pt]
\sum_{k=0}^{(n+1)/2}\frac{(q^{-1};q^2)_k^2 q^{2k}}{(q^2;q^2)_k^2} &\equiv 0 \pmod{\Phi_n(q)^2}, \label{eq:main-3-2}
\end{align}
because $(q^{-1};q^2)_k\equiv 0\pmod{\Phi_n(q)}$ for $(n+1)/2<k\leqslant n-1$.
We now let $\zeta\ne1$ be an $n$-th root of unity, not necessarily primitive.
In other words, $\zeta$ is a primitive root of unity of odd degree $d\mid n$.  If $c_q(k)$ denotes the $k$-th term on the left-hand side of \eqref{eq:main-2-2},
i.e.,
$$
c_q(k)=\frac{(q^{-1};q^2)_k^2 q^{2k}}{(q^2;q^2)_k^2}.
$$
The congruences \eqref{eq:main-2-2} and \eqref{eq:main-3-2}  with $n=d$ imply that
\begin{align*}
\sum_{k=0}^{(d+1)/2}c_\zeta(k)=\sum_{k=0}^{d-1}c_\zeta(k)=0.
\end{align*}
Observe that
$$
\frac{c_\zeta(\ell d+k)}{c_\zeta(\ell d)}
=\lim_{q\to\zeta}\frac{c_q(\ell d+k)}{c_q(\ell d)}
=c_\zeta(k).
$$
We get
$$
\sum_{k=0}^{n-1}c_\zeta(k)=\sum_{\ell=0}^{n/d-1}\sum_{k=0}^{d-1}c_\zeta(\ell d+k)
=\sum_{\ell=0}^{n/d-1}c_\zeta(\ell d) \sum_{k=0}^{d-1}c_\zeta(k)=0,
$$
and
$$
\sum_{k=0}^{(n+1)/2}c_\zeta(k)
=\sum_{\ell=0}^{(n/d-3)/2} c_\zeta(\ell d)\sum_{k=0}^{d-1}c_\zeta(k)+\sum_{k=0}^{(d+1)/2}c_\zeta((n-d)/2+k)=0,
$$
which mean that the sums $\sum_{k=0}^{n-1}c_q(k)$ and $\sum_{k=0}^{(n+1)/2}c_q(k)$ are both divisible
by the cyclotomic polynomial $\Phi_d(q)$.  As this is true for arbitrary divisor $d>1$ of $n$, we conclude that these two sums are both divisible by
\begin{equation*}
\prod_{\substack{d\mid n,\, d>1}}\Phi_d(q)=[n].
\end{equation*}
Namely, the congruences \eqref{eq:main-2-2} and \eqref{eq:main-3-2} are also true modulo $[n]$. The proof then follows from $\gcd([n],\Phi_n(q)^2)=[n]\Phi_n(q)$.
\qed

\section{Proof of Theorem \ref{thm-1}}
The proof is similar to that of Theorem \ref{main-1}.
We first prove the following result.
\begin{align}
\sum_{k=0}^{n-1}\frac{(a^5q, q/a^5, q^4;q^6)_k q^{6k}}{(a^4q^6, q^6/a^4, q^6;q^6)_k}
\equiv 0 \pmod{(1-aq^n)(a-q^n)}\quad\text{if $n\equiv 5\pmod{6}$}.  \label{eq:6-5}
\end{align}
The $r=1$ and $d=6$ case of \eqref{qdk-0} gives
\begin{align*}
\frac{(q^{1-5n};q^6)_k q^{6k}}{(q^6;q^6)_k} &=(-1)^k {(5n-1)/6\brack k}_{q^6} q^{6{(5n-1)/6-k\choose 2}-6{(5n-1)/6\choose 2} }.
\end{align*}
Moreover, we have
\begin{align*}
\frac{(q^{5n+1};q^6)_k}{(q^{4n+6};q^6)_k} &=\frac{(q^{4n+6k+6};q^6)_{(n-5)/6}}{(q^{4n+6};q^6)_{(n-5)/6}},\\[5pt]
\frac{(q^4;q^6)_k}{(q^{6-4n};q^6)_k} &=\frac{(q^{6-4n+6k};q^6)_{(2n-1)/3}}{(q^{6-4n};q^6)_{(2n-1)/3}}.
\end{align*}
It follows that
\begin{align}
&\sum_{k=0}^{n-1}\frac{(q^{1-5n}, q^{5n+1}, q^4;q^6)_k q^{6k}}{(q^{6-4n}, q^{4n+6},  q^6;q^6)_k} \notag\\[5pt]
&\quad=\sum_{k=0}^{n-1}(-1)^k {(5n-1)/6\brack k}_{q^6} q^{6{(5n-1)/6-k\choose 2}-6{(5n-1)/6\choose 2}} \frac{(q^{4n+6k+6};q^6)_{(n-5)/6}(q^{6-4n+6k};q^6)_{(2n-1)/3} }
{(q^{4n+6};q^6)_{(n-5)/6}(q^{6-4n};q^6)_{(2n-1)/3} }.  \label{eq:sum}
\end{align}
Since $(q^{4n+6k+6};q^6)_{(n-5)/6}(q^{6-4n+6k};q^6)_{(2n-1)/3}$ is a polynomial in $q^{6k}$ of degree $(5n-7)/6<(5n-1)/6$, by the identity \eqref{eq:qbino},
we see that the right-hand side of \eqref{eq:sum} vanishes. This proves that the left-hand side of  \eqref{eq:6-5} is equal to $0$ for $a=q^{-n}$ or $a=q^n$.
That is, the congruence \eqref{eq:6-5} holds. Finally, letting $a\to 1$ in \eqref{eq:6-5}, we are led to \eqref{main-6-5}.

Similarly we can prove \eqref{main-6-1}. Here we merely give its parametric generalization:
 \begin{align*}
\sum_{k=0}^{n-1}\frac{(a^5/q,  q^{-1}/a^5, q^{-4};q^6)_k q^{6k}}{(a^4q^6,  q^6/a^4, q^6;q^6)_k}
\equiv 0 \pmod{(1-aq^n)(a-q^n)}\quad\text{if $n\equiv 1\pmod{6}$}.
\end{align*}

\section{More congruences modulo $\Phi_n(q)^2$}
It seems that there are many more similar congruences modulo $\Phi_n(q)^2$. Here we give some such results.

\begin{thm}
Let $n$ be a positive integer. Then
\begin{align*}
\sum_{k=0}^{n-1}\frac{(q,q,q^7;q^9)_k q^{9k}}{(q^9;q^9)_k^3}
&\equiv 0 \pmod{\Phi_n(q)^2}\quad\text{if $n\equiv 2,8\pmod{9}$,} \\[5pt]
\sum_{k=0}^{n-1}\frac{(q^2,q^2,q^5;q^9)_k q^{9k}}{(q^9;q^9)_k^3}
&\equiv 0 \pmod{\Phi_n(q)^2}\quad\text{if $n\equiv 4,7\pmod{9}$,} \\[5pt]
\sum_{k=0}^{n-1}\frac{(q^4,q^4,q;q^9)_k q^{9k}}{(q^9;q^9)_k^3}
&\equiv 0 \pmod{\Phi_n(q)^2}\quad\text{if $n\equiv 5,8\pmod{9}$.}
\end{align*}
\end{thm}
\begin{proof}The proof is similar to that of Theorem \ref{thm-1}. Here we just give the parametric generalizations of
these congruences. Modulo $(1-aq^n)(a-q^n)$, for $r=1,2,4$, we have
\begin{align*}
\sum_{k=0}^{n-1}\frac{(a^5q^r,q^r/a^5,q^{9-2r};q^9)_k q^{9k}}{(aq^9,q^9/a,q^9;q^9)_k}
&\equiv 0 \quad\text{if $n\equiv 2r\pmod{9}$,} \\[5pt]
\sum_{k=0}^{n-1}\frac{(a^8q^r,q^r/a^8,q^{9-2r};q^9)_k q^{9k}}{(a^7q^9,q^9/a^7,q^9;q^9)_k}
&\equiv 0 \quad\text{if $n\equiv -r\pmod{9}$.}
\end{align*}
\end{proof}

\begin{thm}
Let $n>9$ be a positive integer. Then
\begin{align*}
\sum_{k=0}^{n-1}\frac{(q^{-1},q^{-1},q^{-7};q^9)_k q^{9k}}{(q^9;q^9)_k^3}
&\equiv 0 \pmod{\Phi_n(q)^2}\quad\text{if $n\equiv 5\pmod{9}$,} \\[5pt]
\sum_{k=0}^{n-1}\frac{(q^{-2},q^{-2},q^{-5};q^9)_k q^{9k}}{(q^9;q^9)_k^3}
&\equiv 0 \pmod{\Phi_n(q)^2}\quad\text{if $n\equiv 2,5\pmod{9}$,}\\[5pt]
\sum_{k=0}^{n-1}\frac{(q^{-4},q^{-4},q^{-1};q^9)_k q^{9k}}{(q^9;q^9)_k^3}
&\equiv 0 \pmod{\Phi_n(q)^2}\quad\text{if $n\equiv 2\pmod{9}$.}
\end{align*}
\end{thm}
\begin{proof}This time the parametric generalizations of
these congruences are as follows. Modulo $(1-aq^n)(a-q^n)$,
\begin{align*}
\sum_{k=0}^{n-1}\frac{(a^7q^{-1},q^{-1}/a^7,q^{-7};q^9)_k q^{9k}}{(a^5q^9,q^9/a^5,q^9;q^9)_k}
&\equiv 0 \quad\text{if $n\equiv 5\pmod{9}$,} \\[5pt]
\sum_{k=0}^{n-1}\frac{(a^8q^{-2},q^{-2}/a^8,q^{-5};q^9)_k q^{9k}}{(a^7q^9,q^9/a^7,q^9;q^9)_k}
&\equiv 0 \quad\text{if $n\equiv 2\pmod{9}$ and $n>9$,} \\[5pt]
\sum_{k=0}^{n-1}\frac{(a^5q^{-2},q^{-2}/a^5,q^{-5};q^9)_k q^{9k}}{(aq^9,q^9/a,q^9;q^9)_k}
&\equiv 0 \quad\text{if $n\equiv 5\pmod{9}$ and $n>9$,} \\[5pt]
\sum_{k=0}^{n-1}\frac{(a^7q^{-4},q^{-4}/a^7,q^{-1};q^9)_k q^{9k}}{(a^5q^9,q^9/a^5,q^9;q^9)_k}
&\equiv 0 \quad\text{if $n\equiv 2\pmod{9}$ and $n>9$.}
\end{align*}
\end{proof}

\section{Concluding remarks and open problems}
In this section we propose several conjectures for further study. Not like before, there is no symmetry in the following two conjectures.
It seems difficult to find the corresponding parametric generalizations.
\begin{conj}\label{conj-3}
Let $n$ be a positive integer with $n\equiv 4,7\pmod{9}$. Then
\begin{align*}
\sum_{k=0}^{n-1}\frac{(q,q^2,q^6;q^9)_k q^{9k}}{(q^9;q^9)_k^3}
\equiv 0 \pmod{\Phi_n(q)^2}.
\end{align*}
\end{conj}

\begin{conj}\label{conj-4}
Let $n$ be a positive integer with $n\equiv 5\pmod{9}$. Then
\begin{align*}
\sum_{k=0}^{n-1}\frac{(q^{-1},q^{-2},q^{-6};q^9)_k q^{9k}}{(q^9;q^9)_k^3}
\equiv 0 \pmod{\Phi_n(q)^2}.
\end{align*}
\end{conj}

There are many similar conjectures.
Let $n>1$ be a positive integer. For any rational number $x$ whose denominator is coprime with $n$, let $\langle x\rangle_n$ denote the least non-negative residue of $x$ modulo $n$.
We would like to propose the following two conjectures.

\begin{conj}\label{conj-first}
Let $n$ be a positive integer with $n\equiv 2\pmod{3}$, and let $m$ be a positive integer with $\gcd(m,n)=1$. If $r$ is an integer satisfying
$0<\langle \frac{r}{3m}\rangle_n\leqslant\frac{2n-1}{3}$, then
\begin{align*}
\sum_{k=0}^{n-1}\frac{(q^{m};q^{3m})_k (q^{r};q^{3m})_k (q^{2m-r};q^{3m})_k q^{3mk}}{(q^{3m};q^{3m})_k^3}
\equiv 0 \pmod{\Phi_n(q)^2}.
\end{align*}
\end{conj}

\begin{conj}\label{conj-second}
Let $n>1$ be a positive integer with $n\equiv 1\pmod{3}$, and let $m$ be a positive integer with $\gcd(m,n)=1$. If $r$ is an integer satisfying
$0<\langle \frac{r}{3m}\rangle_n\leqslant\frac{2n-5}{3}$, then
\begin{align*}
\sum_{k=0}^{n-1}\frac{(q^{-m};q^{3m})_k (q^{r};q^{3m})_k (q^{-2m-r};q^{3m})_k q^{3mk}}{(q^{3m};q^{3m})_k^3}
\equiv 0 \pmod{\Phi_n(q)^2}.
\end{align*}
\end{conj}
Letting $d=3$, $r=1$ and $q\to q^{m}$ in Theorem \ref{main-1}, and noticing that $\Phi_n(q)$ is a factor of $\Phi_n(q^m)$ for $\gcd(m,n)=1$,
we see that Conjectures \ref{conj-first} and \ref{conj-second}
are true for $r=m$ and $r=-m$, respectively.

\vskip 5mm \noindent{\bf Acknowledgments.} The author was partially
supported by the National Natural Science Foundation of China (grant 11771175).


\begin{thebibliography}{99}
\small \setlength{\itemsep}{-.8mm}

\bibitem{Andrews}G.E. Andrews, The theory of partitions, Cambridge University Press, Cambridge, 1998.

\bibitem{CLZ} K.K. Chan, L. Long, and V.V. Zudilin, A supercongruence motivated by the Legendre
family of elliptic curves, Mat. Zametki 88 (2010), 620--624; translation in Math. Notes 88 (2010), 599--602.

\bibitem{Guo-par}V.J.W. Guo, Some $q$-congruences with parameters, Acta Arith., to appear; arXiv: 1804.10963.

\bibitem{GPZ}V.J.W. Guo, H. Pan, Y. Zhang, The Rodriguez-Villegas type congruences for truncated q-hypergeometric functions, J. Number Theory 174 (2017), 358--368

\bibitem{GZ14}V.J.W. Guo, J. Zeng, Some $q$-analogues of supercongruences of Rodriguez-Villegas, J. Number Theory 145 (2014), 301--316.

\bibitem{GuoZu}V.J.W. Guo, W. Zudilin, A $q$-microscope for supercongruences, preprint, March 2018, arXiv:1803.01830.

\bibitem{Mortenson1}E. Mortenson, A supercongruence conjecture of Rodriguez-Villegas for a certain truncated hypergeometric function,
J. Number Theory 99 (2003), 139--147.

\bibitem{NP}H.-X. Ni, H. Pan, On a conjectured $q$-congruence of Guo and Zeng, Int. J. Number Theory 14 (2018), 1699--1707.

\bibitem{RV}F. Rodriguez-Villegas, Hypergeometric families of Calabi--Yau manifolds, in: Calabi-Yau
Varieties and Mirror Symmetry (Toronto, ON, 2001), Fields Inst.
Commun., 38, Amer. Math. Soc., Providence, RI, 2003, pp.~223--231.

\bibitem{SunZH}Z.-H. Sun, Congruences concerning Legendre polynomials, Proc. Amer. Math. Soc. 139 (2011), 1915--1929.

\bibitem{SunZW}Z.-W. Sun, Super congruences and Euler numbers, Sci. China Math. 54 (2011), 2509--2535.

\bibitem{Tauraso1} R. Tauraso, An elementary proof of a Rodriguez-Villegas supercongruence, preprint, 2009, arXiv:0911.4261v1.

\bibitem{Tauraso2} R. Tauraso, Supercongruences for a truncated hypergeometric series, Integers 12 (2012), \#A45.


\bibitem{Zudilin}W. Zudilin, Ramanujan-type supercongruences, J. Number Theory 129 (2009), 1848--1857.
\end{thebibliography}
\end{document}